\documentclass[dvipsnames,11pt,reqno]{amsart}

\usepackage{amssymb,mathtools,amsthm}
\usepackage[a4paper,left=3cm,right=3cm,top=3cm,bottom=3cm]{geometry}
\usepackage[T1]{fontenc}
\usepackage[utf8]{inputenc}
\usepackage{microtype}
\usepackage{enumitem}

\mathtoolsset{showonlyrefs}

\newcounter{n}
\numberwithin{n}{section}
\theoremstyle{plain}
\newtheorem{lemma}[n]{Lemma}
\newtheorem{theorem}[n]{Theorem}
\newtheorem{proposition}[n]{Proposition}

\theoremstyle{definition}
\newtheorem{definition}[n]{Definition}
\newtheorem{remark}[n]{Remark}
\newtheorem*{claim*}{Claim}
\newtheorem*{proposition*}{Proposition}

\usepackage{tikz}
\usepackage{subcaption}

\usepackage{hyperref}
\definecolor{colorlinks}{RGB}{0, 24, 168}
\definecolor{colorcites}{RGB}{124, 10, 2}
\hypersetup{
    colorlinks=true,
    linkcolor=colorlinks,
    citecolor=colorcites,
    urlcolor=colorlinks,
    pdfborder={0 0 0}
}

\renewcommand\phi\varphi
\renewcommand\epsilon\varepsilon
\renewcommand\setminus\smallsetminus

\newcommand\E{\mathbb E}

\newcommand\R{\mathbb R}

\newcommand\Z{\mathbb Z}
\newcommand\V{\mathbb V}

\newcommand\calT{\mathcal T}
\newcommand\calF{\mathcal F}

\newcommand\calP{\mathcal P}

\newcommand\blank{\,\cdot\,}

\newcommand\Hom{\operatorname{Hom}}
\newcommand\Mon{\operatorname{Mon}}
\newcommand\tree{\mathbb{T}}
\newcommand\Var{\operatorname{Var}}
\newcommand\Geom{\operatorname{Geom}}

\begin{document}

\makeatletter
\@namedef{subjclassname@2020}{\textup{2020} Mathematics Subject Classification}
\makeatother

\title{Height function localisation on trees}
\subjclass[2020]{Primary 82B41; secondary 82B30}
\author{Piet Lammers}
\address{Institut des Hautes \'Etudes Scientifiques}
\email{lammers@ihes.fr}
\author{Fabio Toninelli}
\address{Technische Universit\"at Wien}
\email{fabio.toninelli@tuwien.ac.at}
% \date{29 October 2020}
\keywords{Localisation, height functions, trees, statistical mechanics}

\begin{abstract}
We study two models of \emph{discrete height functions}, that is,
models of random integer-valued functions on the vertices of a tree.
First, we consider the \emph{random homomorphism model}, in which
neighbours must have a height difference of exactly one.  The local
law is uniform by definition.  We prove that the height variance of
this model is bounded, uniformly over all boundary conditions
(both in terms of location and boundary heights).
This implies a strong notion of localisation,
uniformly over all extremal Gibbs measures of the system.
For the
second model, we consider directed trees, in which each vertex has
exactly one parent and at least two children.  We consider the locally
uniform law on height functions which are \emph{monotone}, that is,
such that the height of the parent vertex is always at least the height
of the child vertex.  We provide a complete classification of all
extremal gradient Gibbs measures, and describe exactly the
localisation-delocalisation transition for this model.  Typical
extremal gradient Gibbs measures are localised also in this case.
Localisation in both models is consistent with the observation that the Gaussian free field
is localised on trees, which is an immediate consequence of transience of the
random walk.

\end{abstract}

\maketitle

\setcounter{tocdepth}{1}
\tableofcontents

% \makeatletter
% \providecommand\@dotsep{5}
% \makeatother
% \listoftodos[Inline to-do's]
% \relax

\section{Introduction}

\subsection{Height functions as lattice models}

Height functions, in the broadest sense,
are random real- or integer-valued functions on the vertices of a connected
graph $G=(V,E)$.
Let us say for now that $G$ is finite.
Each edge $xy\in E$ has an associated  potential function $V_{xy}$ (often assumed to be convex, symmetric,
and satisfying $V_{xy}(a)\to\infty$ as $|a|\to\infty$)
and the density of a configuration is proportional to
\begin{equation}
\label{eq:gradH}
    \exp\left[-H(h)\right];\quad H(h):=\sum_{xy\in E}V_{xy}(h(y)-h(x)).
\end{equation}
This well-defines the distribution of the \emph{gradient} of the
height function $h$.  We informally think of this construction as follows: each
edge has a spring associated to it, which pulls the heights at its two
endpoints together. The value $V_{xy}(a)$ is then the energy needed to impose the difference $h(y)-h(x)=a$.
One may also introduce boundary conditions by fixing the value of the height on a certain subset $\partial V\subset V$, seen as the ``boundary''. This way, we obtain a well-defined probability distribution of the height profile $h=(h(x))_{x\in V}$ and not just a distribution of height gradients.
One can also make sense of height functions on infinite connected
graphs, through the standard Dobrushin-Lanford-Ruelle formalism
detailed below.

Height functions arise as \emph{lattice models}
once we choose the role of the graph $G$ to be played by the square lattice 
graph $(\Z^d,\E)$ or a variant thereof.
There is a range of height function models which are each of great physical 
importance: examples include
the \emph{dimer models} (which arise from random tilings), the \emph{six-vertex model}
(which is a model of random arrows and relates to \emph{Fortuin-Kasteleyn percolation}), and
the \emph{discrete Gaussian free field} (DGFF) (which has continuous heights and quadratic potential functions, and is also known as the \emph{harmonic crystal}).
Height functions with convex potential functions on lattices were studied in great generality 
in the work of Sheffield~\cite{SHEFFIELD}.

A very intriguing issue for height functions is the so-called
\emph{localisation-delocalisation transition}.  \emph{Localisation}
essentially means that, for any given $x\in V$, the variance
$\Var[h(x)]$ remains bounded uniformly as boundary conditions are
taken further and further away; \emph{delocalisation} means that the
variance grows infinite.  In dimension $d=1$, all measures are
delocalised, with a full description of the scaling limit in terms of
Brownian motion.  In dimension $d=2$, measures are expected to either
localise, or to delocalise with the Gaussian free field as their scaling limit.
Real-valued models should always be delocalised (quite general results
are discussed in~\cite{velenik2006localization,milos2015delocalization}).
Integer-valued models in dimension $d=2$ are known to localise at low
temperature due to the Peierls argument~\cite{BW}, and they
are expected to delocalise above the temperature of the roughening
transition $T_R$.
There exist many results for specific models in
this direction~\cite{frohlich1981kosterlitz,kenyon2001dominos,giuliani2020non,duminil2020macroscopic,glazman2021uniform,duminil2020delocalization,lammers2022height,lammers2021delocalisation,aizenman2021depinning}, but a general theory of the roughening
transition is still lacking.  In dimension $d\geq 3$, all reasonable height
functions are expected to localise at any temperature.  The heuristics
behind this is that, in the harmonic case $V_{xy}(a)=a^2$ (which
corresponds to the DGFF on $G$), the
variance of $h(x)$ is simply the Green's function $G(x,x)$ of the simple
random walk on $\mathbb Z^d$, which is finite for $d\ge3$ due to
transience. This has been proved rigorously for the discrete Gaussian
model~\cite{Frohlich+Simon+Spencer-1976,Frohlich+Israel+Lieb+Simon-1978} and the
solid-on-solid model~\cite{Bricmont+Fontaine+Lebowitz-1982} in dimension three
and higher, as well as for the uniformly random Lipschitz function in
sufficiently high dimension~\cite{peled2017high}. Our work focuses on
the case where $G$ is a tree. Since trees correspond, in a sense,
to a high-dimensional setting, one typically expects localisation.
Partial results in this
direction, for special boundary conditions, already exist in the
literature (see Subsection~\ref{sec:lit}).

\subsection{Height functions on trees: an apparent contradiction}
\label{sec:appcontr}
If $G$ is a tree, a description of the distribution \eqref{eq:gradH}
is particularly easy (and the construction can be carried out directly
on infinite trees): in that case, the height increments on the edges
are independent, with the increment $h(y)-h(x)$ having a density
proportional to $\exp[-V_{xy}]$.  If all potential functions are the
same, then the measure \eqref{eq:gradH} has the property that the
variance of $h(y)-h(x)$ is proportional to the graph distance from $x$
to $y$.  Thus, we might think that the variance of $h(x)$ blows up as
we put boundary conditions further and further away: the measure would
then be \emph{delocalised}.  On the other hand, we mentioned above
that one expects localisation when the random walk on $G$ is
transient: since this is the case on trees with degree strictly larger
than two, one expects that the variance should be uniformly bounded: the
model would be \emph{localised}.  One purpose of this article is to
address this apparent contradiction.  In fact, we will study a
(not-exactly-solvable) model of integer-valued heights on
trees, and we will show that the model is localised in a strong
sense: the height variance is bounded, uniformly on the choice of the
vertex and (for finite trees) on boundary conditions. For infinite
trees, localisation holds for any extremal Gibbs measure. In
particular, the Gibbs measure described above, with i.i.d.\ gradients on
edges, is \emph{not} extremal; see Remark \ref{rem:nocontradiction}
below.
The notion of extremality really provides the key to resolving the apparent contradiction,
and is closely related to the \emph{reconstructability} of the height distribution given a single sample
(see for example~\cite[Theorem~7.12]{G11} or~\cite[Lemma~3.2.3, Statement~1]{SHEFFIELD}).
We conjecture that some ideas apply in great generality: we
expect that, under a strict convexity assumption on the potential, all
height functions on trees are uniformly localised, that is, the
variance is bounded uniformly over the distance to the boundary and
the choice of boundary conditions.
For real-valued height functions the analogous 
uniform localisation result is an immediate corollary of the Brascamp-Lieb inequality.
We will mention exactly
the parts in the proof that are missing for general potentials
in the integer-valued setting.

\subsection{Models under consideration and relations to previous works}
\label{sec:lit}

The first model under consideration in this article is the model of graph homomorphisms,
in which the height at neighbours must always differ by exactly one.
The local law is uniform.
The model of graph homomorphisms has been considered by several mathematicians,
and in particular also on non-lattice graphs.
General, rudimentary inequalities were obtained in~\cite{benjamini2000random}:
for example, it is proven that the variance of the height difference grows at most linearly in the distance between vertices.
That article also studies the behaviour of the model on a tree under zero boundary conditions at the leaves.
Further bounds on the \emph{maximum} of the random graph homomorphism are derived in~\cite{benjamini2000upper},
by an analogy with the DGFF.
The same problem is revisited in~\cite{benjamini2007random},
which, in particular, relates the bounds on the maximum to the degree of the graph
(in relation to its size).
These first three papers consider height functions in the very first sense presented in the introduction:
on finite graphs, and without boundary.
In two later articles, Peled, Samotij, and Yehudayoff~\cite{peled2013grounded,peled2013lipschitz} consider uniformly random $K$-Lipschitz functions 
on a tree, subjected to zero boundary conditions on a designated boundary set $\partial V\subset V$.
In either case, the authors obtain extremely tight bounds on the pointwise fluctuations of the random integer-valued functions.

Our presentation differs from the results quoted so far, in the sense that we have as objective 
to bound the fluctuations of the model uniformly over all boundary conditions.
This makes our setup less symmetric: for example, under both free and zero boundary conditions,
the law of the model is invariant under a global sign flip.
With non-zero boundary conditions this is clearly not the case.

The second model under consideration lives on directed trees
which have the property that each vertex has exactly one parent and at least two children.
We consider the locally uniform law on height functions which are 
\emph{monotone} in the sense that the height at the parent always dominates the height at the child.
We give a complete classification of Gibbs measures,
and describe exactly the (nontrivial) localisation-delocalisation transition.
It is possible for this model to delocalise,
even though we shall argue that this behaviour is atypical.
The potential associated to this model is convex but not strictly convex,
which invalidates our heuristic which (we believe) rules out delocalisation on trees.
There are several works on potentials which are not (strictly) convex
and which have delocalised extremal Gibbs measures~\cite{henning2021coexistence,henning2021existence,henning2019gradient}.
% In this case, it is possible for the height functions to delocalise, even on trees.
This includes the solid-on-solid model which has the convex potential function $V_\beta(a):=\beta|a|$ associated to it.
It is easy to see that such a potential cannot induce uniform localisation (although ruling out general localisation is harder):
if one forces the height function to equal some non-constant boundary height function $b$ on the complement of a finite set $\Lambda\subset V$,
then one can approximate a continuum model by multiplying the boundary height function $b$ by an extremely large integer $k$.
This necessarily blows up the pointwise variance of the height function at certain vertices in $\Lambda$.
This suggests that the strictly convex framework introduced below is indeed the correct framework for 
proving uniform localisation.

\section{Definitions and main results}

\subsection{Graph homomorphisms on trees}

 We adopt several notations and concepts from the work of Georgii~\cite{G11}
 and Sheffield~\cite{SHEFFIELD}, but we will try to be as self-contained as possible.

\begin{definition}[Graph homomorphisms]
    An \emph{irreducible tree} is a tree $\tree=(\V,\E)$ with minimum degree at least three.
    A \emph{height function} is a map $h:\V\to\Z$,
    and a \emph{graph homomorphism} is a height function $h:\V\to\Z$
    such that $h(y)-h(x)\in\{\pm1\}$ for any edge $xy\in\E$.
    Write $\Hom(\tree,\Z)$ for the set of graph homomorphisms on $\tree$.
\end{definition}

Write $\Lambda\subset\subset\V$ to say that $\Lambda$ is a finite subset of $\V$.

\begin{definition}[Uniform graph homomorphisms]
    \label{def:uni_hom}
    Write $\calF$ for the natural product sigma-algebra on $\Omega:=\V^\Z\supset\Hom(\tree,\Z)$.
    Write $\calP(\Omega,\calF)$ for the set of probability measures on $(\Omega,\calF)$.
    For any $\Lambda\subset\subset\V$ and $b\in\Hom(\tree,\Z)$,
    let $\gamma_\Lambda(\blank,b)\in\calP(\Omega,\calF)$ denote the uniform probability measure on the set
    \begin{equation}
        \left\{
            h\in\Hom(\tree,\Z) : h|_{\V\setminus\Lambda} = b|_{\V\setminus\Lambda}    
        \right\}.
    \end{equation}
    The probability measure $\gamma_\Lambda(\blank,b)$ is called the 
    \emph{local Gibbs measure in $\Lambda$ with boundary height function $b$}.
    The family $\gamma:=(\gamma_\Lambda)_{\Lambda\subset\subset\V}$ forms a \emph{specification}.
    A \emph{Gibbs measure} is a measure $\mu\in\calP(\Omega,\calF)$ which is supported on $\Hom(\tree,\Z)$
    and which satisfies the \emph{Dobrushin-Lanford-Ruelle (DLR) equations} in the sense that $\mu\gamma_\Lambda=\mu$
    for any $\Lambda\subset\subset\V$.
    In this case we say that (the distribution of) $h$ is \emph{locally uniform} in $\mu$.
    A Gibbs measure is said to be \emph{tail-trivial} if it is trivial on the 
    \emph{tail sigma-algebra} $\calT\subset\calF$.
    This sigma-algebra is realised as the intersection over $\Lambda\subset\subset\V$ of all sigma-algebras
    $\calT_\Lambda$ which make $h(x)$ measurable for all $x\in\V\setminus\Lambda$.
    Such measures are also called \emph{extremal}.
\end{definition}

\begin{definition}[Gradient measures]
    \label{def:uni_hom_grad}
    Let $\calF^\nabla$ denote the smallest sub-sigma-algebra of $\calF$ 
    which makes all the height differences $h(y)-h(x)$ measurable.
    A measure $\mu\in\calP(\Omega,\calF^\nabla)$ is called a \emph{gradient Gibbs measure}
    whenever it is supported on $\Hom(\tree,\Z)$ and satisfies the DLR equation $\mu\gamma_\Lambda=\mu$
    for each $\Lambda\subset\subset\V$.
    Extremal gradient Gibbs measures are gradient Gibbs measures which are
    trivial on the gradient tail-sigma-algebra $\calT^\nabla:=\calT\cap\calF^\nabla$.
\end{definition}

Note that we have defined gradient Gibbs measures $\mu$ on the same state space $
\Omega$ as for Gibbs measures. A  point $h
\in\Omega$ is a height function; however, the height $h(x)$ at a given vertex is not a $\mu$-measurable function if $\mu$ is  a gradient measure:
only height differences are.
In other words, the measure $\mu$ returns as samples height functions which are measurable \emph{up to an additive constant}. 
For the validity of the previous definition, observe that $\Hom(\tree,\Z)\in\calF^\nabla$,
and that each probability kernel $\gamma_\Lambda$ restricts to a probability kernel 
from $\calF^\nabla$ to $\calF^\nabla$.

\begin{remark}[Convex interaction potentials]
    \label{rem:convex_V}
  Uniform graph homomorphisms (and also the uniform monotone functions
  of the next section) are examples of random height functions which arise from a convex
  interaction potential. This means that the local Gibbs measure
  $\gamma_\Lambda(\cdot,b)$ can be written as a tilted version 
  of the counting measure, with the Radon-Nikodym derivative proportional to
  \begin{equation}
    \label{BG}
    \textstyle
    1_{\{h|_{\mathbb V\setminus \Lambda}=b|_{\mathbb V\setminus \Lambda}\}}\exp\left[-\sum_{xy\in \mathbb E(\Lambda)} V(h(y)-h(x))\right],
  \end{equation}
and where the potential function $V$ is convex.
Here $\mathbb E(\Lambda)$ denotes the set of edges having at least one endpoint in $\Lambda$.
For graph homomorphisms, $V$ is defined on odd integers, and
it equals $V(a)=\infty\times{\bf 1}_{|a|>1}$.
Sheffield~\cite{SHEFFIELD} developed a broad theory 
for the study of height functions arising from convex interaction potentials,
which applies to our setting.
\end{remark}

\begin{definition}[Localisation-delocalisation]
\label{def:loc}
An extremal gradient Gibbs measure $\mu$ is said to be
\emph{localised} if it is the restriction of a
(non-gradient) Gibbs measure to the gradient sigma-algebra,
and \emph{delocalised} otherwise.
If such a Gibbs measure exists, then it may be chosen to be extremal as well
(this is a straightforward consequence of extremality of the gradient measure $\mu$).
Lemma~8.2.5 of~\cite{SHEFFIELD} asserts that if
$\mu$ is an extremal Gibbs measure, then at every vertex $x\in \V$ the law of
$h(x)$ is log-concave, and therefore has finite moments of all
order.
This implies immediately
that the height fluctuations of
(localised) extremal Gibbs measures are finite.
\end{definition}
  
\begin{theorem}[Uniform localisation for graph homomorphisms]
    \label{thm:loc}
    Let $\tree=(\V,\E)$ denote an irreducible tree with a distinguished vertex $x\in\V$.
    Then the locally uniform graph homomorphism is uniformly localised, in the sense that:
    \begin{enumerate}
        \item Any gradient Gibbs measure is the restriction of a Gibbs measure to $\calF^\nabla$,
        \item There exists a universal constant $C$ such that any extremal Gibbs measure $\mu$ satisfies \[\Var_\mu [h(x)]\leq C.\]
    \end{enumerate}
\end{theorem}
With respect to Definition~\ref{def:loc}, 
the localisation is said to be \emph{uniform} in this theorem because of the universal
upper bound in the second statement.
The constant $C$ is also uniform with
respect to the choice of the irreducible tree $\mathbb T$.
Note
that the result is false if we drop the requirement that the tree has
minimum degree at least three:
if the tree is $\mathbb Z$, then graph
homomorphisms are just trajectories of simple random walks.
In this case,
most gradient Gibbs measures are delocalised: these are i.i.d.\ laws on interface gradients, where each $h(x+1)-h(x)$ is a $\pm1$-valued Bernoulli random variable of some fixed parameter $p\in[0,1]$.
Such measures are localised if and only if $p\in\{0,1\}$,
in which case the configuration is completely frozen.

Let us remark that Theorem~\ref{thm:loc} can be extended (with the same proof) to other convex potentials $V$
with a uniform lower bound on the second derivative, conditional on a technical assumption that is explained in Remark~\ref{rem:gen} below.

\begin{remark}
 \label{rem:nocontradiction}
The uniform localisation statement of the theorem
         is not in contradiction with the fact that the gradient Gibbs
         measure $\mu$ where all height gradients along edges of the
         tree are i.i.d.\ symmetric random variables with values in
         $\{-1,+1\}$ is such that $\Var_\mu[h(x)-h(y)]$ is
         proportional to the graph distance between $x$ and $y$. In
         fact, it turns out that $\mu$ may be viewed as an \emph{actual} Gibbs measure
         (that is, not just a \emph{gradient} Gibbs measure), but also that it is \emph{not
           extremal}; this resolves the apparent contradiction
         mentioned in Section~\ref{sec:appcontr}. Both facts can be
         shown by defining the following random variable.
         Let $A_k$ denote the (empirical) average of $h$ 
         on the vertices at distance $k$ from some distinguished root vertex.
         (This sequence is always well-defined, even though it is not $\calF^\nabla$-measurable.
         It is obviously $\calF$-measurable.)
         Then $(A_k)_k$ is a martingale with independent increments
         (the increments are given by averaging the coin flips on the edges 
         leading to the new vertices).
         Its limit $A_\infty$ is called a \emph{height offset variable}
         (first defined in the work of Sheffield~\cite{SHEFFIELD},
         and discussed in more detail below).
         Note that $A_\infty$ is $\calT$-measurable, but not $\calT^\nabla$-measurable.
         However, its fractional part $\tilde A_\infty:=A_\infty-\lfloor A_\infty\rfloor$ is also $\calT^\nabla$-measurable.
         Its distribution is called the \emph{spectrum} of $A_\infty$.
         The distributions of the increments of the martingale are known explicitly,
        and therefore it is easy to see that the spectrum cannot be a Dirac mass.
        Thus, $\mu$ is nontrivial on the gradient tail-sigma-algebra,
        that is, the measure is not extremal:
         see~\cite[Sections~8.4 and~8.7]{SHEFFIELD} for details.
\end{remark}

\subsection{Monotone functions on directed trees}

\begin{definition}[Directed trees and monotone functions]
    By a \emph{directed tree} we mean a directed tree $\vec\tree=(\V,\vec \E)$
    such that each vertex has exactly one parent and at least one child.
    The parent of a vertex $x\in\V$ is written $p(x)$.
    If we write $xy\in\vec\E$, then $x$ is the child and $y$ the parent.
    A \emph{monotone function} is a function $f:\V\to\Z$
    such that $f(y)-f(x)\geq 0$ for all $xy\in\vec\E$.
    Write $\Mon(\vec\tree,\Z)$ for the set of 
    monotone functions.
\end{definition}

Notice that although our interest is in trees where each vertex has at least 
two children,
we do not impose this requirement in the definition.
In fact, the classification of Gibbs measures works also for trees in which 
some vertices have a single child.

\begin{definition}[Uniform monotone functions]
    For any $\Lambda\subset\subset\V$ and $b\in\Mon(\vec\tree,\Z)$,
    let $\gamma_\Lambda(\blank,b)\in\calP(\Omega,\calF)$ denote the uniform probability measure on the set
    \begin{equation}
        \left\{
            h\in\Mon(\vec\tree,\Z) : h|_{\V\setminus\Lambda} = b|_{\V\setminus\Lambda}    
        \right\}.
    \end{equation}
    All other definitions carry over from Definitions~\ref{def:uni_hom} and~\ref{def:uni_hom_grad}.
\end{definition}

Monotone functions fit the framework of convex interaction potentials 
(Remark~\ref{rem:convex_V})
by setting $V(a):=\infty\times {\bf 1}_{a<0}$
in the sum in~\eqref{BG},
where the sum now runs over \emph{directed} edges.

\begin{definition}[Flow, flow measures]
  \label{def:flow}
    A \emph{flow} $\phi:\vec\E\to(0,\infty]$ on a directed tree $\vec\tree=(\V,\vec\E)$
    is a map which satisfies
    \begin{equation}\textstyle
        \label{eq:flow_condition}
        \phi_{xp(x)}=\sum_{y:\,yx\in\vec\E}\phi_{yx}\quad\forall x\in\V.    
    \end{equation}
    The associated \emph{flow measure} is the gradient Gibbs measure $\mu_\phi$
    defined such that \[(h(y)-h(x))_{xy\in\vec\E}\] is an independent family of random variables,
    and such that
    \(
        h(y)-h(x)\sim\Geom(\phi_{xy}),
    \)
    where $\Geom(\alpha)$ is the distribution on $\Z_{\geq 0}$
    such that the probability of the outcome $k$ is proportional to $e^{-k\alpha}$.
\end{definition}

\begin{theorem}[Complete classification of gradient Gibbs measures]
    \label{thm:classification}
    The map $\phi\mapsto\mu_\phi$ is a bijection from the set of flows
    to the set of extremal gradient Gibbs measures.
\end{theorem}

The classification relies on the very particular combinatorial properties
of the model; a similar classification is not expected in general.

\begin{theorem}[Localisation criterion]
    \label{thm:flow_localised}
    A flow measure $\mu_\phi$ is localised
    if and only if
    \begin{equation}
        \label{eq:crit}
        \sum_{y\in A(x)} e^{-\phi_{yp(y)}}<\infty
    \end{equation}
    for some $x\in\V$,
    where $A(x)\subset\V$ denotes the set of ancestors of $x$.
\end{theorem}

\begin{remark}
    \begin{enumerate}
\item The summability condition is independent of the choice of $x\in\V$,
because $|A(x)\Delta A(y)|<\infty$ for any $x,y\in\V$.

\item
The previous
theorem implies that gradient Gibbs measures are localised in most
natural cases.  Indeed, if $\vec\tree$ is a regular tree in which each
vertex has $d\geq 2$ children, and if $\phi$ is invariant under the
automorphisms of $\vec\tree$ which preserve the generations of the
tree, then $\phi$ increases by a factor $d$ with each older generation.  In
that case, it is clear that the localisation criterion is satisfied.
Since it is natural to impose extra symmetries on $\phi$, we say that
uniform monotone functions are \emph{typically localised} on trees.

\item
On the other hand, a (pathological) example of a delocalised flow
measure is obtained by taking $\phi$ to be a flow that equals $p>0$ on
a directed infinite ray $\mathcal R$ of $\vec\tree$ and zero
elsewhere. In this case, the height increments are i.i.d.\ $\Geom(p)$
as one walks along $\mathcal R$, while the height equals $-\infty$ on vertices
outside of $\mathcal R$. Strictly speaking this example is not allowed by Definition~\ref{def:flow}
because the flow should be strictly positive, but one can easily
construct flows that approximate the finite constant 
$p$ along $\mathcal R$ in the ancestor direction,
and such that it is strictly positive
(but small) on the remaining edges.

\item
If $\mathbb T=\mathbb Z$, then graph homomorphisms and monotone height functions are in bijection (simply via a $45^\circ$ rotation of the graph of the height function).
This mapping fails when $\mathbb T$ is a non-trivial tree.
In fact, the results above show that one can have delocalisation for uniform monotone surfaces but not for uniform homomorphisms.

    \end{enumerate}
\end{remark}

Our results so far concern (gradient) Gibbs measures on infinite
trees. We address now the following finite-volume question. Let $
\vec\tree=(\V,\vec\E) $ be a regular directed tree where each vertex
has $d\ge 2$ children.
Fix a vertex $v\in\V$, and let $D^n$ denote the
set of descendants of $v$ (including $v$ itself), up to distance $n-1$
from $v$. Let also $\partial D^n$ denote the set of descendants of $v$
that are at distance exactly $n$ from it. Let $\gamma_n$ be the
uniform measure on monotone functions on $D^{n+1}$, subject to the
boundary conditions $h(v)=0$ and $h|_{ \partial D^n}\equiv- n
$. Theorem~\ref{thm:flow_localised} already suggests that the height
will be localised under $\gamma_n$, with bounded fluctuations as
$n\to\infty$. However, it is not \emph{a priori} clear whether the typical
height profile will be dictated by the boundary height $0$, which is enforced at the vertex $v$,
or the boundary height $-n$, enforced at the descendants at distance $n$,
or
if the height profile will somewhat smoothly interpolate between them.
This question is
answered by the following result.
\begin{theorem}
  \label{th:finite_volume}
  Sample $h|_{D^{n+1}}$ from $\gamma_n$.  For any fixed descendant $x$ of
  $v$, $h(x)$ tends to $0$ in probability, as
  $n\to\infty$. Moreover, 
  for some suitably chosen constant $c(d)$ depending only on $d$,
  the event that $h(x)=0$ for all descendants of $v$ at distance at most $n-c(d)\log n$ from it has probability $1+o(1)$ as
  $n\to\infty$.
\end{theorem}

The theorem holds true more generally if the boundary height at $ \partial D^n$
is fixed to some value $-\lfloor a n\rfloor$ for $a>0$.
The fact that
the effect of the boundary condition at the single
vertex $v$ prevails against the effect of the boundary condition at
the $d^n$ vertices of $ \partial D^n$, except at distance $O(\log n)$ from
the latter, is perhaps surprising at first sight.
On the other hand, note that a
positive fraction of the vertices of $D^n$ is actually at distance $O(1)$
from $\partial D^n$.
This property sets the theorem apart from the case that $d=1$,
in which case the height profile is linear.
% Note also that if the number of children is $d=1$ instead (i.e. for
% uniform monotone heights on $\mathbb Z$) the statement of the Theorem does not hold: the height profile is
% linear, with the average of $h(x)$ equal to $-d(x,v)$.

\section{Uniform localisation of graph homomorphisms}

\begin{definition}
    Let $p:\mathbb Z\to[0,1]$ denote a probability mass function,
    whose support consists of a consecutive sequence of either odd or even integers.
    The distribution $p$ is called \emph{$\alpha$-strongly log-concave} for some $\alpha\geq 1$
    whenever
    \[
        p(k)^2\geq\alpha p(k-2)p(k+2)
        \quad \forall k\in\Z.
    \]
\end{definition}

The following follows immediately from the previous definition.
\begin{proposition}
    \label{proposition:product_concavity}
    Let $p,q:\mathbb Z\to[0,1]$ denote probability mass functions
    which are supported on the even or the odd integers,
    and whose supports are not disjoint.
    If $p$ and $q$ are $\alpha$- and $\beta$-strongly log-concave respectively,
    then the normalized probability mass function $\frac1Zpq$ is $\alpha\beta$-strongly log-concave.
\end{proposition}

\begin{remark}
  \label{rem:var}
  If a probability mass function $p$ is $\alpha$-strictly log-concave
  for some
  $\alpha>1$, then it has finite moments of all orders and its variance is bounded by a constant $C(\alpha)$ that
  depends only on $\alpha$.
  To see this, suppose that $p$ takes its absolute maximum at some $k_0\in \mathbb Z$.
  Strict convexity implies
  \[
    \log p(k)-\frac12(\log p(k-1)+\log p(k+1))\ge \beta=:\frac12\log \alpha>0,
  \]
  so that
  \[
    p(k_0+\ell)\le p(k_0)e^{-\beta(\ell^2-1)}\quad\forall \ell\in\mathbb Z
  \]
  and finiteness of all moments is obvious.
  The variance of any real-valued random variable $X$ 
  may be written $\Var X=\inf_z\mathbb E[(X-z)^2]$,
  and the uniform bound depending only on $\alpha$ follows by choosing $z=k_0$.
\end{remark}

If $p$ and $q$ are probability mass functions,
then write $p*q$ for their convolution, which is also a probability mass function.
Write $X:\Z\to[0,1]$ for the probability mass function defined by
$2X(k)={\bf 1}_{|k|=1}$.

\begin{definition}
    Let $\lambda\in[3,4]$ denote the unique real root of
    \(
        x^3 - 3 x^2 - x - 1
    \).
\end{definition}

\begin{lemma}
    \label{lem:lambdasquaredlambda}
    If $p$ is $\lambda^2$-strongly log-concave,
    then $p*X$ is $\lambda$-strongly log-concave.
\end{lemma}

\begin{remark}
  \label{rem:gen}
  Suppose that $W$ is a convex potential function,
  and let $Y=e^{-W}$.
  If there exists a constant $\lambda(W)>1$
  such that $p*Y$ is $\lambda(W)$-strongly log concave 
  whenever $p$ is $\lambda(W)^2$-strongly log concave,
  then all results in this section generalise to the height 
  function model induced by the potential $W$.
  Following~\cite[Section~6.2]{saumard2014log},
  it is natural to expect that such a constant $\lambda(W)$
  exists whenever $Y$ is $\alpha$-strongly log-concave 
  for some $\alpha>1$. 
  Note that $Y$ is $\alpha$-strongly log-concave 
  for some $\alpha>1$
  if and only if $W$ is uniformly strictly convex,
  that is, a potential of the form $k\mapsto\epsilon k^2+V(k)$
  where $\epsilon>0$ and where $V$ is any convex function.
\end{remark}

\begin{proof}[Proof of Lemma~\ref{lem:lambdasquaredlambda}]
    Suppose that $p$ is supported on odd integers.
    Let $q:=p*X$.
 By shifting the domain of $q$ if necessary (which does not modify its log-concavity properties),
 it suffices to show, without loss of generality, that
    \[
        q(0)^2\geq\lambda q(-2)q(2),
    \]
    which is equivalent to
    \[
        (p(-1)+p(1))^2\geq\lambda(p(-3)+p(-1))(p(1)+p(3)).
    \]
    If $p(-1)$ or $p(1)$ equals zero, then it is easy to see that the
    right hand side must be zero, and we are done.  Otherwise, let
    $\alpha:=p(1)/p(-1)$.  By multiplying $p$ by a constant if
    necessary, we suppose that $p(-1)=\alpha\lambda^2$ and
    $p(1)=\alpha^2\lambda^2$.  Now $\lambda^2$-strong log-concavity of
    $p$ implies that $p(-3)\leq 1$ and $p(3)\leq\alpha^3$.  For the
    sake of deriving $\lambda$-strong log-concavity, we may take these
    inequalities for equalities.  Thus, it now becomes our goal to
    demonstrate that
    \[
        (1+\alpha)^2\alpha^2\lambda^4\geq \lambda(1+\alpha\lambda^2)\alpha^2(\lambda^2+\alpha)
    \]
    for any $\alpha>0$, for the given value of $\lambda$.
    Rearranging gives
    \[
        \alpha^2-\frac{\lambda^4-2\lambda^3+1}{\lambda^3-\lambda^2}\alpha+1\geq 0.
    \]
    This parabola in $\alpha$ does not take negative values for positive $\alpha$ whenever
    \[
        -\frac{\lambda^4-2\lambda^3+1}{\lambda^3-\lambda^2}\geq-2,
    \]
    that is,
    \[
        (1-\lambda)(\lambda^3-3\lambda^2-\lambda-1)\geq 0.
    \]
    This inequality holds trivially true when $\lambda$ is a zero of the factor on the right.
\end{proof}

\begin{lemma}
    \label{lemma:finite_loc}
    Let $\tree=(\V,\E)$ denote an irreducible tree and $x\in\V$ some distinguished vertex.
    Then for any $\Lambda\subset\subset\V$ and $b\in\Hom(\tree,\Z)$,
    the law of $h(x)$ in the measure $\gamma_\Lambda(\blank,b)$ is $\lambda^2$-strongly log-concave.
\end{lemma}

\begin{proof}
    The proof idea is as follows.  For each $y\in\V$ distinct from
    $x$, there is a unique edge incident to $y$ which points in the
    direction of $x$.  We will let $p_y$ denote the law of $h(y)$ in the
    measure $\gamma_\Lambda(\blank,b)$, except that to define this
    measure we pretend that this special edge is not there.  (In
    particular, when $y\neq z$, $p_y$ and $p_z$ are marginals of
    distinct measures on height functions.)  This yields a recursive
    relation (see~\eqref{eq:recursion} below) which, in combination
    with Proposition~\ref{proposition:product_concavity} and
    Lemma~\ref{rem:gen}, allows us to derive the asserted
    concentration.  It is important to observe that the Markov
    property is very strong on trees: removing edges induces
    independence of the model on the connected components of the
    remaining graph.

    Fix $\Lambda$ and $b$. If $x\not\in\Lambda$, then the
    distribution of $h(x)$ is a Dirac mass at $b(x)$, which means in
    particular that it is $\lambda^2$-strongly log-concave, and we are
    done. Let
    \[
        \Lambda_k:=\{y\in\V:d(x,y)<k\},% ;
        % \quad
        % \partial\Lambda_k:=\{y\in\V:d(x,y)=k\}, 
    \]
    and write $\E_k$ for the set of edges which are incident to
    $\Lambda_k$.  Write $\mu_k$ for the uniform measure on the set of
    functions $h\in\Z^{\V\setminus\Lambda_k}$ such that $h$ equals $b$
    on the complement of $\Lambda\cup\Lambda_k$, and such that
    $h(z)-h(y)\in\{\pm1\}$ for all $yz\in\E\setminus\E_k$.  Remark
    that $\mu_0=\gamma_\Lambda(\blank,b)$.
    For each $y\in\V$, we let $p_y$ denote
    the law of $h(y)$ in $\mu_k$ where $k:=d(x,y)$, see 
    Fig.~\ref{fig:1} (this is equivalent to the informal definition of $p_y$ given above).  It suffices to prove the claim that $p_y$ is
    $\lambda^2$-strongly log-concave for any $y\in\V$ (the lemma
    follows by taking $y=x$).
\begin{figure}
\centering
\includegraphics[scale=.6]{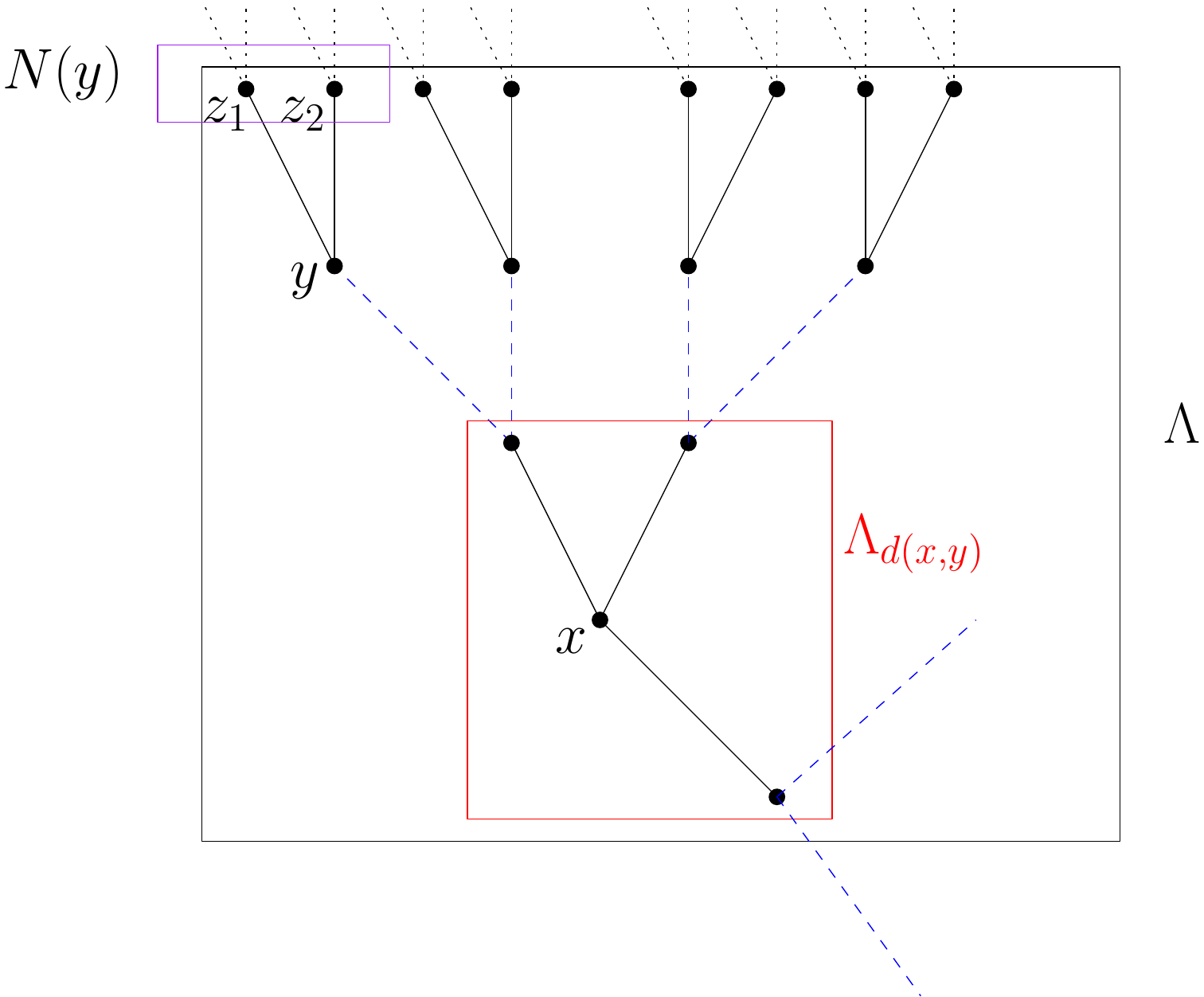}
\caption{A graphical explanation of the procedure in the proof of
  Lemma~\ref{lemma:finite_loc}, in a case where $\mathbb T$ is a
  regular tree of degree $3$. Only the portion of the tree inside
  $\Lambda\cup \Lambda_{d(x,y)}$ is drawn; outside, the configuration
  coincides with $b$. In defining the measure $\mu_k$ with $k=d(x,y)$
  (this measure is defined on the complement of $\Lambda_k$), the blue
  dashed edges are removed and the tree splits into a finite number of
  connected components. Under $\mu_k$, the height at $y$ is
  independent of the height on the components that do not contain
  $N(y)=\{z_1,z_2\}$
  (that is, the connected components not containing $y$). Its law $p_y$ is obtained by conditioning
  two independent random variables having distributions $p_{z_1}*X$ and $p_{z_2}*X$
  to be equal, which gives~\eqref{eq:recursion}. Here, the definition of $X$ comes from the
  fact that height the gradients along the edges $yz_1$ and $yz_2$
  take only values $\pm1$ with equal likelihood.}\label{fig:1}
\end{figure}

     We prove the statement by induction on the distance $k:=d(x,y)$.
      If $d(x,y)\ge d_0$ for a suitable choice of $d_0$ depending on $\Lambda$, then $\Lambda_k\supset \Lambda$. In this case,  $\mu_k$ is the Dirac measure on $b|_{\V\setminus\Lambda_k}$
      and $p_y$ is also a Dirac measure, hence $\lambda^2$-strongly log-concave. %  By choosing $n$ large enough,
    % we have $\Lambda\subset\Lambda_n$,
    % so that $\mu_n$ is the Dirac measure on $b|_{\V\setminus\Lambda_n}$.
    % Therefore the claim is true for $y\not\in\Lambda_n$.
    We prove the full claim by inductively lowering $d(x,y)$ from $d_0$ to $0$.
Fix $y$ with $d(x,y)=d<d_0$. 
    Let $N(y)$ denote the neighbours of $y$ which are at  distance $d+1$ from $x$.  Note that $|N(y)|\geq 2$ because the minimum degree of the tree is at least three.
    % Fix $y\in\Lambda_n$.
    % If $y\not\in\Lambda$ then the claim is trivial.
    Moreover, it is easy to see (as explained in the caption of Fig.~\ref{fig:1}) that
    \begin{equation}
        \label{eq:recursion}
        p_y=\frac1Z \prod_{z\in N(y)}(p_z*X).
    \end{equation}
    By induction, all the functions $p_z$ are $\lambda^2$-strongly log-concave,
    so that Proposition~\ref{proposition:product_concavity} and Lemma~\ref{rem:gen} imply that $p_y$ is $\lambda^2$-strongly log-concave.
\end{proof}

The proof of Theorem~\ref{thm:loc} now follows from abstract arguments which may appear technical but which are more or less standard.
First, we introduce height offset variables.

\begin{definition}[Height offset variables,~\cite{SHEFFIELD}]
    A \emph{height offset variable} is a random variable
    $H:\Omega\to\R\cup\{\infty\}$ which is $\calT$-measurable 
    and satisfies $H(h+a)=H(h)+a$ for all $h\in\Omega$ and $a\in\Z$.
    The variable $H$ is called a \emph{height offset variable}
    for some gradient Gibbs measure $\mu$ whenever $\mu(H<\infty)=1$.
\end{definition}

Remark that $\{H<\infty\}\in\calT^\nabla$ because the event is unchanged if
we add a constant to each height function,
which means that it depends only on the height
gradients.
This is perhaps surprising because $H$ is not $\calT^\nabla$-measurable. Height
offset variables capture the following idea.  Let $\mu$ denote a gradient Gibbs
measure. This measure may be turned into a non-gradient measure by a
procedure called \emph{anchoring}, that is, by defining $h(x):=0$ for
some fixed vertex $x\in\V$. This preserves the DLR equations
$\mu\gamma_\Lambda=\mu$ for all $\Lambda\subset\subset\V$ which do not
contain $x$, but not for sets $\Lambda$ that do contain $x$. Thus, we
would like instead to \emph{anchor at the tail}. This is exactly what
height offset variables do. For height functions on $\Z^d$, height
offset variables, when they exist, often arise as the $n\to\infty$
limit of the empirical average of the height function over concentric
balls of radius $n$.

\begin{lemma}
    If there exists a height offset variable $H$ for some gradient Gibbs measure $\mu$,
    then $\mu$ is the restriction of a Gibbs measure to the gradient sigma-algebra.
\end{lemma}

\begin{proof}
    The proof, including more context and background, may be found in Sheffield~\cite{SHEFFIELD},
    but we include it here for completeness.
    Define a new measure $\tilde\mu$ as follows:
    to sample from $\tilde\mu$, sample first $h$ from $\mu$,
    then output $\tilde h:=h-\lfloor H(h)\rfloor$.
    Since $H$ is a height offset variable, $\tilde h$
    is invariant under replacing $h$ by $h+a$ for any $a\in\Z$.
    In other words, $\tilde h$ is invariant under the choice 
    of the height function $h$ to represent its gradient.
    This means that $\tilde\mu$ is well-defined as a probability measure in
    $\calP(\Omega,\calF)$, and $\tilde h(x)$ is a $\tilde\mu$-measurable function.
    Since $H$ is tail-measurable, $\tilde\mu$ is a Gibbs measure:
    the DLR equations hold because the measure is anchored at the tail, as mentioned above.
    Finally, plainly, the measure $\tilde\mu$ restricted to the gradient sigma-algebra is just $\mu$.
\end{proof}

We are now ready to prove Theorem~\ref{thm:loc}.

\begin{proof}[Proof of Theorem~\ref{thm:loc}]
    The choice of the vertex $x$ is fixed throughout.
    Start with the first part.
    Let $\mu$ denote a gradient Gibbs measure.
    Our goal is to construct a height offset variable.
    Define $\Lambda_k$ as in the proof of Lemma~\ref{lemma:finite_loc}.
    Let $\calF_k^\nabla$ denote the smallest sigma-algebra
    which makes $h(z)-h(y)$ measurable for all $y,z\in\V\setminus\Lambda_k$.
    Note that $(\calF_k^\nabla)_{k\geq 0}$ is a backward filtration
    with $\cap_k\calF_k^\nabla=\calT^\nabla$.
    Write $X_k$ for the random variable defined by
    \[
        X_k:=\mu(h(x)-h(a_k)|\calF_k^\nabla)-(h(x)-h(a_k)),
      \]
      where each $a_k$ is chosen in the complement of $\Lambda_k$.
      It is straightforward to work out that this definition is 
      independent of the choice of $a_k$,
      precisely because the height gradient outside $\Lambda_k$ is measurable 
      with respect to $\calF_k^\nabla$.
        Thus, $X_k$ tells
      us, given the current height function $h$, how much we expect
      $h(x)$ to increase once we resample all heights at a distance
      strictly below $k$ from $x$.  Remark that $X_0=0$.  It suffices
      to demonstrate that $X_k$ converges almost surely to some
      $\calF^\nabla$-measurable limit $X_\infty$,
      because \[H:=X_\infty+h(x)\] is the desired height offset
      variable in that case (it is clear that it is indeed tail measurable).

    For fixed $n$, 
    the sequence $(X_k-X_n)_{0\leq k\leq n}$ is a backward martingale in the backward 
    filtration $(\calF_k^\nabla)_{0\leq k\leq n}$.
    By orthogonality of martingale increments,
    this means that
    \[
        \|X_n\|_2^2
        =
        \|X_k\|_2^2
        +
        \|X_k-X_n\|_2^2
    \]
    in $L^2(\mu)$ for any $0\leq k\leq n$.
    But by Lemma~\ref{lemma:finite_loc}, \[\|X_n\|_2^2
    =\int\Var_{\gamma_{\Lambda_n}(\blank,b)}[h(x)]d\mu(b)
    \le C(\lambda^2)
    \]
    independently of $n$, where the constant $C(\lambda^2)$ comes from Remark~\ref{rem:var}.
    Therefore $(X_k)_{k\geq 0}$ is Cauchy in $L^2(\mu)$ and converges to some limit 
    $X_\infty$ in $L^2(\mu)$.
    This convergence can be upgraded to almost sure convergence by using Doob's upcrossing inequality or 
    the backward martingale convergence theorem,
    applied to the backward martingale $(X_k-X_\infty)_{k\geq 0}$.
    
    Now focus on the second part.
    Let $\mu$ denote an extremal Gibbs measure.
    Define $Y_k$ by $X_k+h(x)$;
    this is simply the expected value of $h(x)$ given the values of $h$ on 
    the complement of $\Lambda_k$.
    By the first part,
    $Y_k$ converges almost surely to some limit $H$.
    Moreover, since $\mu$ is extremal, $H$ is almost surely equal to some 
    constant $a\in\R$.
    The tower property implies that $\mu(h(x))=a$.
    With Fatou's lemma, we may estimate the variance by
    \[
        \Var_\mu [h(x)]
        \leq\lim_{k\to\infty}\mu((h(x)-Y_k)^2)
        =\lim_{k\to\infty}\|X_k\|_2^2.   
    \]
    But $\|X_k\|_2^2\leq C(\lambda^2)$ as observed above, which proves
    the second part of the theorem with $C=C(\lambda^2)$.  (In fact,
    the inequality on the left in the display turns into an equality
    with these \emph{a posteriori} bounds, but this is not important
    to us.)
\end{proof}

\section{Complete classification of Gibbs measures for monotone height functions}
\label{sec:complete}

Let $\mu$ denote an extremal gradient Gibbs measure throughout this section.
For $x\in\V$, let $D(x)\subset\V$ denote the set of descendants of $x$,
which by convention includes $x$.
For $k\geq 0$, let $D_k(x)$ denote the descendants which are at a distance $k$ from $x$,
and let $E_k(x)$ denote the set of edges $yz\in\vec\E$ with $y\in D_k(x)$.
This means that $E_0(x)$ contains one edge, which connects $x$ with its parent $p(x)$.
Define a sequence of random variables $(X_k)_{k\geq 0}$
by
\[\label{eq:Xk}
    X_k:=\min_{yz\in E_k(x)}h(z)-h(y).
\]
For convenience we shall write $\nabla h|_V$ 
for the map
\[
    V\times V\to\Z,\,(y,z)\mapsto h(z)-h(y)
\]
for any $V\subset\V$.

\begin{lemma}
  \label{lem:definetti}
    Under $\mu$, the sequence $(X_k)_{k\geq 0}$ is i.i.d.\ with distribution
    $\Geom(\alpha)$ for some $\alpha\in(0,\infty]$.
    Moreover, this sequence is independent of $\nabla h|_{\V\setminus D(x)}$.
\end{lemma}

\begin{proof}
   For lightness, we drop the argument $x$ in the notations $D(x)$ and $D_k(x)$.
    We first claim that conditional on $\nabla h|_{\V\setminus D}$,
    the distribution of the sequence $(X_k)_{k\geq 0}$ satisfies the following two criteria:
    \begin{enumerate}
        \item It is exchangeable,
        \item Conditional on $X_0+\dots+X_n=a$, the distribution of $(X_0,\dots,X_n)$
        is uniform in
        \[
            \{
                x\in \Z_{\geq 0}^{\{0,\dots,n\}}:x_0+\dots+x_n=a
            \}.
        \]
    \end{enumerate}
    In fact, the second criterion clearly implies the first, which is
    why we focus on proving it.  First, let
    $D^n:=\cup_{0\leq k<n}D_k$.  Write $\mu'$ for the measure $\mu$
    conditioned on $\nabla h|_{\V\setminus D^n}$.  Since $\mu$ is a
    gradient Gibbs measure and since $\mu'$ is only conditioned on the complement 
    of $D^n$, $\mu'$ still satisfies the DLR equation
    $\mu'\gamma_{D^n}=\mu'$.
    Suppose now that we also condition
    $\mu'$ on $\nabla h|_{D_k}$ for every $0\leq k<n$. This means that
    we are adding the information on the height difference within each
    of sets $D_k$, \emph{but not on the height gradients between a
      vertex in $D_k $ and a vertex in $D_{k'}$ for $k\ne k'$}. 
      Observe that this means that the gradient of the height function $h$ is completely determined
      if we also
    added the information of the variables $X_0,\dots,X_{n}$.
    The random variables $X_0,\dots,X_{n}$
    are nonnegative, and their sum can be determined from the
    information we conditioned on (that is $\nabla h|_{\V\setminus D^n}$ and  $\nabla h|_{D_k}$ for every $0\leq k<n$).  The values of these random
    variables are not constrained in any other way.  Since the
    specification induces local uniformity, the second criterion
    follows.

    Continue working in the measure $\mu$ conditioned on
    $\nabla h|_{\V\setminus D}$.  De Finetti's theorem implies that
    the sequence $(X_k)_{k\geq 0}$ is i.i.d.\ once we condition on
    information in the tail of $(X_k)_{k\geq 0}$.  Since $\mu$ is
    extremal, this tail is trivial, which means that the sequence
    $(X_k)_{k\geq 0}$ is i.i.d..  Moreover, it is easy to see that the
    second criterion implies that the distribution of $X_0$ is
    $\Geom(\alpha)$ for some $\alpha\in(0,\infty]$.  In
      fact, taking $n=1$ and letting $U(x):=\log \mathbb P(X_0=x)$,
      the second criterion implies that $U(x)+U(a-x)$ is constant for $0\le x\le a$,
      so that $U(x+1)-U(x)=U(a-x)-U(a-x-1)$ for $0\le x<a$. Since $a$ is arbitrary, $U$ is affine and therefore $X_0$ is a geometric random variable.

    To prove
      the full lemma, we must demonstrate that this parameter $\alpha$
      does not depend on the information in
      $\nabla h|_{\V\setminus D(x)}$ that we conditioned on.  Indeed,
      if $\alpha$ did depend on this information, then this would
      imply again that the tail sigma-algebra of $(X_k)_{k\geq 0}$ is
      not trivial, contradicting extremality of $\mu$.
\end{proof}

\begin{lemma}
\label{lem:definetti_flow}
  There exists a map $\phi:\vec\E\to(0,\infty]$ such that under the extremal
  gradient measure $\mu$,
    the gradients $(h(y)-h(x))_{xy\in\vec\E}$ are independent with
    distribution
    $h(y)-h(x)\sim\Geom(\phi_{xy})$.
    Moreover, $\phi$ is a flow.
  \end{lemma}

\begin{proof}
    Let, as above, $p(x)$ denote the parent of a vertex $x$.
    The previous lemma implies that there exists some map $\phi:\vec\E\to(0,\infty]$
    such that for each $x\in\V$, we have:
    \begin{enumerate}
        \item $h(p(x))-h(x)\sim\Geom(\phi_{xp(x)})$,
        \item $h(p(x))-h(x)$ is independent of $\nabla h|_{\V\setminus D(x)}$.
    \end{enumerate}
    This implies readily that the family $(h(y)-h(x))_{xy\in\vec\E}$
    is independent.  It suffices to demonstrate that $\phi$ is
    necessarily a flow.  The flow condition~\eqref{eq:flow_condition}
    for a fixed vertex $x\in\V$ follows immediately from the definition of
    the probability kernel $\gamma_{\{x\}}$, which induces local
    uniformity.  In particular, it is important to observe that the
    minimum of $n$ independent random variables having the
    distributions $(\Geom(\alpha_k))_{1\leq k\leq n}$ respectively, is
    precisely $\Geom(\alpha_1+\dots+\alpha_n)$, so that the sequence $(X_k)_k$ is i.i.d.\ if and only if $\phi$ is a flow.
\end{proof}

\begin{proof}[Proof of Theorem~\ref{thm:classification}]
  By the previous lemma it suffices to demonstrate that for each flow
  $\phi$, the measure $\mu_\phi$ is an extremal gradient Gibbs
  measure.  The reasoning at the end of the previous proof implies
  that $\mu_\phi$ satisfies $\mu_\phi\gamma_{\{x\}}=\mu_\phi$ for any
  $x\in\V$. In other words, for any vertex $x$, the measure $\mu_\phi$
  is invariant under the ``heat bath'' Glauber update where the value
  of the height at the single vertex $x$ is resampled according to the specification,
  and conditional on the heights of all other vertices.
  On the other hand, the state space of configurations
  coinciding with some fixed boundary condition on the complement of a fixed finite domain,
  is connected under
  updates at a single vertex.
  This follows from a standard application of the Kirszbraun theorem,
  see for example~\cite[Proof of Lemma~13.1]{LamTas}.
  It is well-known that this implies the DLR equation $\mu_\phi\gamma_\Lambda=\mu_\phi$
  for fixed $\Lambda$,
  simply because the measure on the left may be approximated by 
  composing $\mu_\phi$ with many probability kernels of the form 
  $(\gamma_{\{x\}})_{x\in\Lambda}$.
  In other words, $\mu_\phi$ is a gradient Gibbs measure.
  It suffices to show that
  it is also extremal.  If $\mu_\phi$ was not extremal, then
  the previous lemma implies that it decomposes as a combination of
  other flow measures.  But since a geometric distribution cannot be
  written as the non-trivial convex combination of geometric
  distributions, we arrive at a % the parameter for the geometric this leads to a
  contradiction.
\end{proof}

\begin{proof}[Proof of Theorem~\ref{thm:flow_localised}]
    Recall that every extremal gradient Gibbs measure is a flow measure.
    Let $\phi$ denote a flow and consider the flow measure $\mu_\phi$.
    Fix $x_0\in\V$, and define a sequence $(x_k)_{k\geq 0}\subset\V$
    by $x_{k+1}:=p(x_k)$.
    Suppose first that the localisation criterion is satisfied, that is, that the
    sum~\eqref{eq:crit} converges.
 We observe that $        h(x_n)-h(x_0)    
$ is the sum of $n$ independent Geometric random variables with their respective parameters given by $(\phi_{x_kp(x_k)})_{0\le k<n}$.
    A random variable $X\sim\Geom(\alpha)$ satisfies
    \[
        \E(X)=\frac{e^{-\alpha}}{1-e^{-\alpha}}.    
    \]
    Since $\phi_{x_kp(x_k)}$ is increasing in $k$
    and $\phi_{x_0p(x_0)}>0$, we observe that the sequence
    \[
        h(x_n)-h(x_0)    
    \]
    converges almost surely as $n\to\infty$ if the localisation
    criterion is satisfied. (By looking at the Laplace transform, it is
    also easy to see that the sequence $ h(x_n)-h(x_0) $ diverges
    almost surely otherwise,
    but this fact is not used in this proof.)  If this sequence converges, then
    $\lim_{n\to\infty}h(x_n)$ is a height offset variable, so that
    $\mu_\phi$ is localised.

    Finally, we must prove the converse statement.
    Assume that the
    sum in the localisation criterion diverges;
    our goal is now to prove delocalisation,
    which is achieved through a slightly different route.
    By Definition~\ref{def:loc}, it suffices to prove that
    \[
      \lim_{\Lambda\uparrow\V} \int
      \Var_{\gamma_\Lambda(\blank,b)}[h(x_0)]d\mu(b)=\infty.\] The
    integral is increasing in $\Lambda$
    due to orthogonality of martingale increments.  Write
    $\Lambda=\Lambda_{n,k}:=D^{n+k}(x_n)\setminus\{x_n\}\subset\V$,
    which by definition contains the descendants of $x_n$ from
    generation $1$ to $n+k-1$.  We prove a lower bound on the quantity
    in the previous display, by taking first $k\to\infty$ and then
    $n\to\infty$.  For fixed $n$, we can essentially repeat the
    arguments in the proof of Lemma~\ref{lem:definetti} to see that
    the following claim is true.

    \begin{claim*}
      For $\mu$-almost every $b$,
      the \emph{quenched} law of
      $\nabla h|_{D^{n+1}(x_n)}$ in $\gamma_{\Lambda_{n,k}}(\blank,b)$
      converges weakly as $k\to\infty$ to the law of 
      $\nabla h|_{D^{n+1}(x_n)}$ in the unconditioned measure
      $\mu$.
%     \begin{multline}
%       \label{eq:claim}
%    \text{ for } \mu\text{-almost every }  b, \text{ the \emph{quenched} law of }
%     \nabla h|_{D^{n+1}(x_n)} \text{ in  } \gamma_{\Lambda_{n,k}}(\blank,b)
% \text{    converges weakly as }\\ k\to\infty \text{ to the law of }
%     \nabla h|_{D^{n+1}(x_n)} \text{ in the unconditioned measure } \mu.  
%     \end{multline}
  \end{claim*}
    
\begin{proof}[Proof of the claim]\renewcommand{\qedsymbol}{}
    For each fixed vertex $x\in D^{n+1}(x_n)$,
      introduce the sequence of random variables $(X_\ell)_{\ell\geq 0}$ as
      above.
      Note that if $x$ is at distance $r\ge0$ from $x_n$, then
      under $\gamma_{\Lambda_{n,k}}(\blank,b)$, the random variable $X_\ell$ is deterministic
      whenever $\ell>n+k-r$.
      Each of these sequences $(X_0,\dots,X_{n+k-r})$ has the same
       exchangeability and local uniformity properties as in the proof of Lemma~\ref{lem:definetti}.
       In addition we know that, under the unconditioned law $\mu$,
      the empirical average of $X_0,\dots,X_\ell$ tends almost surely to a constant
      $\alpha$ as $\ell\to\infty$. Since $\mu$ is extremal, it follows
      that for $\mu$-almost every $b$, the quenched law of
      the empirical average of
      $X_0,\dots,X_{n+k-r}$ under the conditional measure $\gamma_{\Lambda_{n,k}}(\cdot,b)$ tends as $k\to\infty$ to a
      Dirac mass at $\alpha$.  This implies that for every $x\in D^{n+1}(x_n)$, the quenched law of $(X_\ell)_{\ell\geq 0}$
      converges almost surely to the law of a sequence of i.i.d.\ Geom($\alpha$), independent of $\nabla h|_{\mathbb V\setminus D(x)}$.
      This readily implies the claim.
\end{proof}

      %  This implies convergence to
    % independent geometric random variables for the height differences
    % of edges, and if the flow parameters do not match the flow
    % parameters of our original measure $\mu$, then there is a
    % contradiction stemming from extremality of $\mu$.
    In turn, the claim
    % convergence of the local law of $h$ in the quenched setting
    implies that
    \begin{multline}
        \lim_{k\to\infty}
        \int\Var_{\gamma_{\Lambda_{n,k}}(\blank,b)}[h(x_0)]d\mu(b)
        \geq
        \Var_\mu[h(x_n)-h(x_0)]
        \\=
        \sum_{m=0}^{n-1}
            \frac{e^{-\phi_{x_mx_{m+1}}}}{(1-e^{-\phi_{x_mx_{m+1}}})^2}
        \geq
        \sum_{m=0}^{n-1}
        e^{-\phi_{x_mx_{m+1}}}
        \to_{n\to\infty}\infty
    \end{multline}
    by divergence of the sum in the localisation criterion. This proves the theorem.
\end{proof}

\section{Monotone functions on finite trees}

This section contains the proof of Theorem~\ref{th:finite_volume}.
Recall that each vertex has $d$ children.
Recall also that $D_n(v)$ denotes the set of $n$-th generation descendants 
of some fixed vertex $v$,
and that $D^{n}(v):=\cup_{0\leq k<n}D_k(v)$.
Let $\gamma_{n,k}$ denote the uniform measure on monotone height functions on the set
$D^{n+1}(v)$, with boundary conditions $h(v)=0$ and $h|_{D_n(v)}\equiv-k$ for some fixed $k\ge0$.

\begin{lemma}
  \label{comblemma}
  Let $z_1,\dots,z_d$ denote the children of $v$. For any $i=1,\dots,d$, we have
  \[
    \gamma_{n,k}(h(z_i)=0)\ge 1-\left(1-\frac1{k+1}\right)^{d^{n-2
      }}.
    \]
\end{lemma}
\begin{proof}
  Fix $i$.
  Under $\gamma_{n,k}$, the height restricted to the descendants of
  $z_i$ is independent of the height on the rest of the graph. To
  match the notations of Section~\ref{sec:complete}, write $x$ instead
  of $z_i$. For $k=0,\dots,n-1$, define the random variable $X_k$ and
  the set of directed edges $E_k(x)$ as in Section~\ref{sec:complete}.
  We know from the proof of Lemma~\ref{lem:definetti} that the
  sequence $X_0,\dots,X_{n-1}$ is exchangeable. Note that
  $X_0=h(v)-h(x)$ and that $X_{n-1}$ is the minimum height gradient
  between one of the $d^{n-1}$ ancestors of $x$ in $D_{n}(v)$ and
  their parent, which belongs to $D_{n-1}(v)$. Given the height on $D_{n-2}(v)$, the height  on the vertices in $D_{n-1}(v)$ are independent and each one is uniformly distributed in an interval of size at most $k+1$. Therefore,
  \[
    \gamma_{n,k}(X_{n-1}=0)\ge 1-\left(1-\frac1{k+1}\right)^{d^{n-2}}. 
  \]
  Since $X_0$ has the same distribution as $X_{n-1}$, the lemma follows.
\end{proof}

\begin{proof}
  [Proof of Theorem~\ref{th:finite_volume}]
  The measure $\gamma_n$ is nothing but $\gamma_{n,n}$ in the notation of Lemma~\ref{comblemma}.
  From Lemma~\ref{comblemma} we have that with $\gamma_{n,n}$ probability at least
  \[
1-d\left(1-\frac1{n+1}\right)^{d^{n-2}},
\]
we have 
$h(z)=0$ for each of the $d$ children of $v$. Conditioned on this event,
the measure $\gamma_{n,n}$ restricted $D^n(v)\setminus\{v\}$ is just the $d$-fold product measure $\gamma_{n-1,n}^{\otimes d}$. Iterating the argument,
the probability that $h(x)=0$ for every $x\in D^n(v)$ at distance up to $m$ from $v$ has $\gamma_n$-probability at least
\[
1-\sum_{j=1}^m d^j\left(1-\frac1{n+1}\right)^{d^{n-j-1}}.
\]
Given that $d\ge 2$, it is easy to see that this quantity is $o(1)$ as $n\to\infty$,
as soon as
$m\le n-c(d)\log n$ for large enough $c(d)$.
\end{proof}

\section*{Acknowledgments}
The authors thank Utkir Rozikov and the anonymous referees
for a careful reading of the manuscript and many useful suggestions for improvement.

P.\ L.\ was supported by the ERC
grant CriBLaM,
and thanks F.\ T.\ and the TU Wien for their hospitality during
several visits.
F.\ T.\ gratefully acknowledges financial support of the Austria Science
Fund (FWF), Project Number P 35428-N.

\bibliographystyle{amsalpha}
\bibliography{references.bib}

\end{document}